\documentclass[11pt]{article}
\usepackage{amssymb, graphicx, parskip,ulem}

\usepackage{lipsum}%%lipsum package produces fake text for testing (below)
\usepackage[utf8]{inputenc}
\usepackage{titlesec}
\usepackage{enumitem}
\usepackage{amssymb,amsmath,amsthm}
\usepackage{graphicx}
\usepackage{todonotes}
\usepackage[utf8]{inputenc}
\usepackage[english]{babel}
\usepackage{float}
\usepackage{multibib}
\usepackage{appendix}
\usepackage[margin=2cm]{geometry}
\usepackage{hyperref}
\usepackage{todonotes}
\usepackage{ulem,cancel}

\makeatletter
\newcommand{\chapterauthor}[1]{%
  {\parindent0pt\vspace*{-25pt}%
  \linespread{1.1}\large\scshape#1%
  \par\nobreak\vspace*{35pt}}
  \@afterheading%
}
\makeatother

\numberwithin{equation}{section}
\newtheorem{theorem}{Theorem}[section]
\newtheorem*{theorem*}{Theorem}
\newtheorem{corollary}[theorem]{Corollary}
\newtheorem{lemma}[theorem]{Lemma}

\theoremstyle{definition}
\newtheorem{remark}[theorem]{Remark}

\newcites{intro,urn,bp,ce,con}%
{References,
References,
References,
References,
References}

\renewcommand{\d}{{\rm d}}
\usepackage{mathtools}
\newcommand{\R}{\right}

\newcommand{\F}{\left}

\newcommand{\E}{\mathbb{E}}

\newcommand{\e}{\mathrm{e}}

\newcommand\conindis{\stackrel{\mathclap{\mathrm{d}}}{\rightarrow}}

\title{A moment approach to the law of large numbers for supercritical branching Markov processes}
\author{Christopher B. C. Dean\thanks{Department of Statistics, University of Warwick, Coventry, CV4 7AL, UK. Email: \texttt{\{Christopher.B.C.Dean\}, \{Emma.Horton\}@warwick.ac.uk}} \and János Engländer\thanks{Department of Mathematics, University of Colorado, Boulder, CO 80309, USA. Email: \texttt{janos.englander@colorado.edu}} \and Emma Horton$^*$}
\date{}

\begin{document}
\maketitle

\begin{abstract}
We offer a new proof of the classical law of large numbers for a general class of branching Markov processes based on the asymptotic behaviour of the moments developed in \cite{bmoments, gonzalez2022erratum}. Moreover, we show that the law of the limiting random variable, that is the almost sure limit of the classical additive martingale, is completely determined by its moments. 

\bigskip


\end{abstract}

\section{Introduction}
 Understanding the asymptotic behaviour of moments has been a central topic of interest in the study of branching processes \cite{bmoments, meatman, Klenke, bmoments2, PSDBP-moments, durhamCMJ}. Moment asymptotics provide a powerful tool for analysing these processes, often revealing information about genealogical structure \cite{felix_genealogies, M2few}, scaling limits \cite{felix-CRT}, and long-term dynamics \cite{Yaglom, Pedro_Yaglom, bCLT}. The focus of this article is to show that  moment asymptotics can be used to obtain a classical law of large numbers (LLN) for a general class of branching Markov processes. Moreover, we show that the limiting random variable is completely determined by its moments. 
 
 Similar to moment asymptotics, LLNs have been well studied for branching particle systems, \cite{EnglanderHarrisKyprianou2010, NTEbook, Marguet2019} with the proofs often utilising spine decompositions and martingale convergence techniques. In this paper, we present an alternative approach that exploits the moment asymptotics developed in \cite{bmoments}.

Before we introduce the class of processes which we will work with, let us introduce some general notation. Throughout, we let $E$ be a Polish space and $\dagger \notin E$ a cemetery (absorbing) point. We let $B(E)$ denote the space of real-valued bounded functions, $B^+(E)$ denote the space of non-negative real bounded functions and $B_1(E)$, (resp. $B_1^+(E)$) denote those functions in $B(E)$, (resp. $B^+(E)$) that are bounded by unity. We extend the definition of $f \in B(E)$ to $E \cup \{\dagger\}$ by setting $f(\dagger) = 0$. 

We let $M(E)$ denote the space of atomic measure with non-negative integer total mass on $E$. For a measure $\mu$ and a function $f \in B(E)$, we write 
\[
  \mu[f] := \int_E f(x)\mu({\rm d}x).
\] 
In the event that $\mu$ has a density with respect to Lebesgue measure, say $\nu$, we slightly abuse notation and write $\nu[f] = \int_E f(x)\nu(x) dx$.

\subsection{Branching Markov processes}\label{subsec:BMP}
We consider a measure-valued stochastic process $X =(X_t, t\geq0)$ given by
\[
X_t := \sum_{i = 1}^{N_t}\delta_{x_i(t)}, \qquad t \ge 0,
\]
whose atoms $\{x_i(t) : i = 1, \dots, N_t\}$ evolve in $E\cup \{\dagger\}$ according to the following dynamics. From an initial position $x \in E$, particles evolve independently in $E$ according to a Markov process $\xi$ with semigroup $(\mathtt P_t)_{t \ge 0}$. When at $y \in E$, particles branch at rate $\gamma(y)$, where $\gamma \in B^+(E)$, at which point the particle gives birth according to the point process $(\mathcal Z, \mathcal P_y)$, where 
\[
 \mathcal Z = \sum_{i = 1}^N \delta_{x_i}.
\] 
We will also often use the notation $\mathcal E_y$ for the expectation associated to $\mathcal P_y$. We also let $\mathbb P_{\mu}$ denote the law of $X$ when initiated from $\mu \in M(E)$ and we write $\mathbb E_\mu$ for the associated expectation operator. 

The process $X$ is uniquely characterised via the non-linear semigroup
\[
  {\rm e}^{-\mathtt v_t[f](x)} := \mathbb E_{\delta_x}[{\rm e}^{-X_t[f]}], \quad f \in B^+(E), t \ge 0, x \in E,
\]
which satisfies the evolution equation
\[
  {\rm e}^{-\mathtt v_t[f](x)} = \mathtt P_t[{\rm e}^{-f}](x) + \int_0^t \mathtt P_s[\mathtt G[{\rm e}^{-\mathtt v_{t-s}[f]}]](x){\rm d}s.
\]
Here 
\[
  \mathtt G[f](x) := \gamma(x)\mathcal E_x\left[\prod_{i = 1}^N f(x_i) - f(x) \right],
\]
is the (non-linear) branching mechanism. We refer to $X$ as a $(\mathtt P, \mathtt G)$-branching Markov process (BMP).

\subsection{Assumptions and main result}
One object that is central to understanding the asymptotic behaviour of branching Markov processes, $X$, is its linear expectation semigroup,
\begin{equation}\label{eq:linear}
  \psi_t[f](x) := \E_{\delta_x}\left[X_t[f]\right], \quad t \ge 0,\, x \in E, \, f \in B(E).
\end{equation} 
Indeed, it is common that the semigroup $\psi$ exhibits the following Perron Frobenius-type behaviour.
\begin{enumerate}[label=(H1),ref=(H1)]
\item \label{H1} There exists $\lambda > 0$, a bounded function $\varphi \in B(E)$ and a probability measure with density $\tilde\varphi$ on $E$, such that, \(\tilde\varphi[\varphi]=1\),
\[
  \psi_t[\varphi](x) = {\rm e}^{\lambda t}\varphi(x), \qquad \tilde\varphi[\psi_t[f]] = \tilde\varphi[f],
\]
and
\begin{equation*}
  \sup_{x \in E, f \in B_1(E)}\big| {\rm e}^{-\lambda t}\varphi(x)^{-1}\psi_t[f](x) - \tilde\varphi[f]\big|
  \to 0, \quad \text{as } t \to \infty.
\end{equation*}
\end{enumerate}
Here $\lambda$ describes the leading order growth rate of the number of particles in the system and, in particular, assumes that the system is {\it supercritical}; $\tilde\varphi$ describes the stationary behaviour of the particles in the system; $\varphi$ can be thought of as an importance function that indicates how favourable each $x \in E$ is for survival of the process when initiated from $x$.

We also introduce an assumption on the moments of the offspring distribution that will not be used in our main result but is necessary in order to present existing results. 

\begin{enumerate}[label=(H2$k$),ref=(H2k)]
\item \label{H2k} Fix $k \ge 1$. We assume that 
\[
  \sup_{x \in E}\mathcal E_x[\mathcal Z[1]^k] < \infty.
\]
\end{enumerate}

This above assumption, along with the assumption on the branching rate, ensures that the $k$-th moment of the process is bounded, uniformly in $x$. Moreover, defining $\psi_t^{(k)}[f](x) := \mathbb E_{\delta_x}[X_t[f]^k]$, $t \ge 0$, $f \in B(E)$ and $x \in E$, the following result was obtained in \cite{bmoments, gonzalez2022erratum}. 
\begin{theorem*}{\cite[Theorem 2]{bmoments, gonzalez2022erratum}}
 Fix $k \ge 2$ and assume that \ref{H1} and \ref{H2k} hold. For $\ell \le k$ and $t \ge 0$, define 
 \begin{equation}\label{eq:moments}
 \Delta_t^{(\ell)} = \sup_{x\in E, f\in B^+_1(E)} \left|{\rm e}^{-\ell \lambda t} \psi_t^{(\ell)}[f](x)
- \,\ell!(\tilde{\varphi}[f])^\ell \varphi(x)L_\ell(x)\right|,
 \end{equation}
where, ${\color{black}L_1 (x)= 1}$ and we define iteratively for $k \ge 2$,
\[
L_k(x)
=\int_0^\infty{\rm e}^{-\lambda ks}{\color{black}\varphi(x)^{-1}} \psi_s\Bigg[ {\gamma}\mathcal{E}_{\cdot}\bigg[ \sum_{[k_1, \dots, k_N]_k^{2}}\prod_{\substack{j = 1 \\ j: k_j > 0}}^N {\color{black}\varphi(x_j)} L_{k_j}(x_j)\bigg]\Bigg](x)\d s, 
\]
with $[k_1, \dots, k_N]_k^2$ denoting the set of all non-negative $N$-tuples $(k_1, \dots, k_N)$ such that $\sum_{i = 1}^N k_i = k$ and at least two of the $k_i$ are strictly positive.
{\color{black}Then, for all $\ell\leq k$
\begin{equation}
\sup_{t\geq0} \Delta_t^{(\ell)}<\infty
\text{ and }
\lim_{t\rightarrow\infty}
\Delta_t^{(\ell)}
=0.
\end{equation}}
\end{theorem*}

Results of this kind have been extensively used in the subsequent study of non-local branching processes including to establish the classical Yaglom limit in the critical case \cite{Yaglom, Pedro_Yaglom}, characterise their genealogical structure \cite{PSDBP-moments, felix_genealogies} and understand their scaling limits \cite{felix-CRT}. Our main result shows how they can be used to offer an alternative proof for well-known law of large numbers. Moreover, our proof enables one to characterise the limiting random variable in terms of its moments.

\begin{theorem}
\label{theorem: WLLN}
    Assume that \ref{H1} holds and further assume that 
    % \begin{equation}\label{eq:H2}
    %     \sup_{x \in E}\mathcal E_x[3^{(N+1)^2}] < \infty.
    % \end{equation}
    \begin{equation}\label{eq:H2}
        \sup_{x \in E}\mathcal E_x[2^N] < \infty.
    \end{equation}
    Then, for any $x \in E$ and $f \in B^+(E)$, we have that
    \begin{equation*}
        \e^{-\lambda t }X_t[f] \conindis W_x \tilde\varphi[f], \quad \text{ as } t \to \infty,
    \end{equation*}
under $\mathbb{P}_{\delta_x}$, where \(W_{x}\) is the almost sure limit of ${\rm e}^{-\lambda t}X_t[\varphi]$. Moreover, \(W_x^k\) is determined by its moments which satisfy,
    for each $k \ge 2$, \(\mathbb{E}[W_{x}^k] = k!\varphi(x)L_k(x)\).
\end{theorem}

\bigskip

\begin{remark}
    Note that this result implies that ${\rm e}^{-\lambda t}X_t(dy) \to W_x\tilde\varphi(dy)$ in law, as $t \to \infty$, which follows from the fact that the theorem holds for all non-negative, bounded and continuous functions $f$.
\end{remark}
% \comj{I would mention that this implies in particular the test function-free form that says $e^{\lambda t}X_t(dx)\to \tilde \varphi X_x$ in law. In fact if one only has the limit for nonnegative, bonded and continuous(!) functions $f$, then that is exactly equivalent to this function-free form. We now have a bit more. }

\bigskip

We also obtain the following corollary by combining \cite[Theorem 5]{bmoments, gonzalez2022erratum} with the above result.

\begin{corollary}\label{cor:WLLN}
Under the assumptions of Theorem \ref{theorem: WLLN}, for any $x \in E$ and $f \in B^+(E)$, we have
\[
{\rm e}^{-\lambda t}\int_0^t X_s[f] \, d s \conindis W_x \tilde\varphi[f], \quad \text{ as } t \to \infty.
\]
\end{corollary}

\bigskip

The main contribution of Theorem \ref{theorem: WLLN} is the fact that the law of \(W_x\) can be completely characterised by its moments. This is notable since these moments have explicit form as stated in \eqref{eq:moments}. Moreover, Theorem \ref{theorem: WLLN} can be generalised to hold under a weaker version of \ref{H1} that assumes the leading eigenvalue is non-simple. In this setting, \(\varphi,\tilde\varphi\) are replaced, respectively, by a collection of bounded eigenfunctions \(\varphi_1,\dots,\varphi_n\in B(E)\) and bounded functionals \(\tilde\varphi_1,\dots,\tilde\varphi_n:B(E)\rightarrow \mathbb{C}\), such that
\[
  \sup_{x \in E, f \in B_1(E)}\big| {\rm e}^{-\lambda t}\psi_t[f](x) - \sum_{i=1}^{k}\tilde\varphi_i[f]\varphi_i(x)\big|
  \to 0, \quad \text{as } t \to \infty.
\]
Again, the moments of \(W_x\) are given explicitly (see \cite{Moments_Chris}). Furthermore, this setting appears in many non-trivial cases (see \cite{Moments_Chris} for examples).

We also briefly comment on the moment assumption \eqref{eq:H2}. While this assumption is quite strong, it is satisfied in several natural examples including growth-fragmentation processes with binary splitting \cite{GF}, neutron transport \cite{NTEbook}.

In the next section we discuss some preliminary results, including the so-called Stieltjes moment problem and some basic theory on the convergence of random measures. Then we present the proof of Theorem \ref{theorem: WLLN}.

\section{Preliminaries}\label{sec:prelims}
\subsection{Moment convergence and the Stieltjes moment problem}\label{subsec:Stieltjes} 
Recall that in the ``Stieltjes moment problem'' one considers non-negative random variables, that is, probability measures on a half-line $[0,\infty)$.
For the moments $s_k:=E[X^k]$ to determine the distribution of $X$ uniquely, it is sufficient that 
\begin{align}\label{eq: Carleman}
\sum_{k=1}^{\infty}s_k^{-1/(2k)}=+\infty.
\end{align}
The condition in \eqref{eq: Carleman} is called the {\it Carleman condition}. Various existing results give simpler criteria that imply \eqref{eq: Carleman}, usually involving bounds for $s_k$, or asymptotics as $k \to \infty$. See, for example, \cite{Lin.moments, Schmudigen.book}. In what follows we will show that there exists $C > 0$ such that 
\[
 s_k \le C^{2k-1}(k!)^2, \quad k \ge 1,
\]
where in our case, $s_k = k!\sup_{x \in E}\varphi(x)L_k(x)$. This and Stirling's approximation then imply that \eqref{eq: Carleman} holds.

% An often used simpler condition (implying \eqref{eq: Carleman}) is that
% that there exists an $M>0$ such that
% \begin{align}\label{eq: simpler.Carleman}
% s_{2n}\le M^n(2n)!
% \end{align}
% holds for $n\ge 1.$ (See e.g. \cite{Schmudigen.book} for \eqref{eq: Carleman} and \eqref{eq: simpler.Carleman}.)
% {\color{blue}EH: this is not what is used later so this will need updating.}

Regarding the convergence of moments we recall the following result.
\begin{theorem}[Theorem 8.6. in \cite{Gut.book}]
Let $X$ and $X_1, X_2, \dots$ be random variables with finite moments of all orders, and suppose that
$E|X_n|^k \to E|X|^k$ as $n \to \infty$, for $k \ge 1$.
If the moments of $X$ determine the distribution of $X$ uniquely, then $X_n\to X$ in law as $n \to \infty.$
\end{theorem}
If one only knows that $s_k:=\lim_{n\to\infty}E|X_n|^k$ exists but not a priori that they correspond to a probability distribution, then it may seem that one needs to establish in addition that the laws of the $X_n$ are tight. However, this comes for free, because by
the Markov inequality, for tightness it is sufficient to check for instance, that
$$\sup_{n\ge 1} E(X_n^2)<\infty,$$
and this we already know as $\lim_{n\to\infty}E|X_n|^2=s_2.$

\subsection{Random measures}\label{subsec:randommeasures}
Below, we give a short summary on the convergence of random measures -- for a reference see \cite{Kallenberg.book}.

(i) Let $\mathcal{X}$ be a Polish space (complete separable metric space), let $C_c(\mathcal{X})$ denote the space of continuous functions with compact support and let $\mathcal{M}_{\mathcal{X}}$ denote the space of finite measures on $\mathcal X$, equipped with the vague topology. This space is  then metrizable with a metric 
 $\rho_{\mathsf{vag}}$ and $(\mathcal{M}_{\mathcal{X}},\rho_{\mathsf{vag}})$ is Polish too. Now consider $X,X_0,X_1,X_2,\dots$ where $X$ and $X_i$ are $\mathcal{M}_{\mathcal{X}}$-valued random elements (random measures). By definition,
 $X_n\to X$ as $n\to\infty$ in law means that for each bounded and continuous map
 $$F:(\mathcal{M}_{\mathcal{X}},\rho_{\mathsf{vag}})\to (\mathbb R,\mathsf{d}),$$ one has $\lim_{n\to\infty}\mathbf{E}(F(X_n))=\mathbf{E}(F(X)).$ (Here $\mathsf{d}$ is the Euclidean metric.) Since 
 $F_{f}(\nu):=\exp\{-\nu(f)\}$ with some $f\in C_c(\mathcal{X})$ is one such map, 
 $X_n\to X$ as $n\to\infty$ in law implies that
 \begin{align}\label{eq: Laplace.tr}
\lim_n \mathcal{L}_{X_{n}}(f):=
\lim_n\mathbf{E}[\exp\{-X_{n}(f)\}]= \mathbf{E}[\exp\{-X(f)\}]=:\mathcal{L}_{X}(f), \ f\in C_c(\mathcal{X}).
 \end{align}
 It is well known though that \eqref{eq: Laplace.tr} is not just necessary but also sufficient for concluding that
 $X_n\to X$ as $n\to\infty$ in law. (Here $\mathcal{L}_{X}$ is the {\it Laplace functional} of $X$.)

 (ii) Next, one can modify the above discussion by  letting $C_b(\mathcal{X})$ denote the space of bounded continuous functions and considering $\mathcal{M}_{\mathcal{X}}$ being equipped with the {\it weak} topology. Then this space is also metrizable with a metric $\rho_{\mathsf{pro}}$ (Prokhorov metric) and $(\mathcal{M}_{\mathcal{X}},\rho_{\mathsf{pro}})$ is Polish too.
 In this case, $X_n\to X$ as $n\to\infty$ in law means that for each bounded and continuous map
 $F:(\mathcal{M}_{\mathcal{X}},\rho_{\mathsf{pro}})\to (\mathbb R,d),$ one has $\lim_{n\to\infty}\mathbf{E}(F(X_n))=\mathbf{E}(F(X)).$
 Similarly to the case of vague topology, it is  now clear that $X_n\to X$  in law implies that
 \begin{align}\label{eq: Laplace.tr.weak}
\lim_n \mathcal{L}_{X_{n}}(f)=\mathcal{L}_{X}(f), \ f\in C_b(\mathcal{X}),
 \end{align}
 and again, it is known that \eqref{eq: Laplace.tr.weak} is in fact not just a necessary but also a sufficient condition for the convergence in law.
 
\section{Proof of Theorem \ref{theorem: WLLN}}
The general strategy of the proof is to use the fact that convergence of moments is sufficient if the limit is determined uniquely by its moments, as discussed in Section \ref{subsec:Stieltjes}. Now, note that for any \(f\in B^+(E)\), we have that
\begin{equation*}
         \e^{-\lambda t}X_t[f] =\tilde \varphi[f]\e^{-\lambda t}X_t[\varphi]+\e^{-\lambda t}X_t[f-\tilde\varphi[f]\varphi].
     \end{equation*}
     Since \(\tilde \varphi[\varphi]=1\), we have that \(\tilde\varphi[f-\tilde\varphi[f]\varphi]=0\) and thus the limit appearing in \eqref{eq:moments} is equal to $0$ for any $k \ge 1$. In particular, taking $\ell=2$, this implies the second term converges in \(L^2\) to 0 as \(t\rightarrow \infty\), hence it remains to show convergence of \(\e^{-\lambda t}X_t[\varphi]\). Since convergence of moments is already given in \cite{bmoments}, it remains to check that the limit \(W_{x}\) is uniquely determined by its moments. As such, we will show that there exists a constant \(C > 0\), such that
    \begin{equation}
    \label{eq:1}
        k! \sup_{x \in E}\varphi(x)L_k(x) \leq C^{2k-1} (k!)^2 \quad \Leftrightarrow \quad \sup_{x \in E}\varphi(x)L_k(x) \leq C^{2k-1}k! .
    \end{equation}
    We prove this using induction. 
    
    The case \(k=1\) is immediate by taking \(C = \Vert \varphi\Vert < \infty\). For \(k\geq 2\), assume that, for each \(1\leq \ell <k\), \eqref{eq:1} holds. Note that,
    \begin{align*}
        \varphi(x)L_k(x) &\leq \sup_{x \in E}\mathcal{E}_{x}\F[\sum_{[k_1,\dots,k_N]^2_k}\prod_{j : k_j > 0}\varphi(x_j)L_{k_j}(x_j)\R]\int_0^{\infty}\e^{-k\lambda s}\psi_s[\gamma](x)\mathrm{d}s \\
        &\leq C_1\sup_{x \in E}\mathcal{E}_{x}\F[\sum_{[k_1,\dots,k_N]^2_k}\prod_{j: k_j > 0}\varphi(x_j)L_{k_j}(x_j)\R]\\
        &\leq C_1C^{2k-2}k!\sup_{x \in E}\mathcal{E}_{x}\F[\sum_{[k_1,\dots,k_N]^2_k}{k \choose k_1, \dots, k_N}^{-1}\R],
    \end{align*} 
    where, since $\lambda>0$, the second inequality is due to \ref{H1}, and the final inequality is due to the inductive hypothesis and that at least two of the \(k_j\) are strictly positive. Importantly, \(C_1\) does not depend on \(k\). To complete the proof, we use the following lemma.
    \begin{lemma}\label{lem:inversemulti}
        There exists a constant \(C>0\) such that
        \begin{equation*}
        \sup_{k\geq 1}\sum_{[k_1,\dots,k_N]_{k}}{k\choose k_1,\dots,k_N}^{-1} \leq C2^N, \quad N\geq 1.
        \end{equation*}
    \end{lemma}
    This, along with \eqref{eq:H2}, then implies that 
\begin{equation}
\label{eq:2}
\sup_{k\geq 2}\sup_{x \in E}\mathcal{E}_{x}\F[\sum_{[k_1,\dots,k_N]^2_k}{k\choose k_1,\dots,k_N}^{-1}\R] \leq C_2,
    \end{equation}
    which yields \eqref{eq:1} with \(C \ge C_1C_2\). 
    %For the case of \(f\in B(E)\), we split \(f\) into its positive and negative parts. {\color{blue}Convergence of the\dots }
    
    \begin{proof}[Proof of Lemma \ref{lem:inversemulti}]
       For \(N,k\geq 1\), define $[k_1, \dots, k_N]_{k, 1}$ to denote the set of $N$-tuples $(k_1, \dots, k_N)$ such that $\sum_{j = 1}^Nk_j = k$ and all the $k_j$ are greater than or equal to $1$. We first claim that there exists a constant \(C>0\) such that
        \begin{equation}\label{eq:step1}
            \sup_{N\geq 1}\sup_{k\geq 1}\sum_{[k_1, \dots,k_N]_{k, 1}}{k\choose k_1,\dots,k_N}^{-1} \leq C.
        \end{equation}
        Let \([k_1, \dots,k_N]_{k, 1}^3\) denote the subset of \([k_1, \dots,k_N]_{k, 1}\) such that at least 3 of the \(k_j\) are greater than 1. Then, 
        \begin{equation}
        \label{eq: lemma bound 2}
            \sup_{N\geq 1}\sup_{k\geq 1}\sum_{[k_1, \dots,k_N]_{k, 1}^3}{k\choose k_1,\dots,k_N}^{-1}\leq   \sup_{N\geq 1}\sup_{k\geq 1}\sum_{[k_1, \dots,k_N]_{k, 1}^3}\frac{2^2(k-N-1)!}{k!} \leq 4,
        \end{equation}
        where we have obtained the first inequality by maximising the summands by taking all but three of the $k_i = 1$, and two of the remaining three \(k_i=2\), and the second inequality by noting that the number of summands is bounded by $k!/(k - N-1)!$. Next note that for all \(k\geq 2\), \(|[k_1, \dots,k_N]_{k, 1}\setminus[k_1, \dots,k_N]_{k, 1}^3|\leq 3kN^2\), since at most \(2\) of the \(k_i\) may differ from \(1\) and fixing one forces the choice of the other. Therefore, by a similar argument to \eqref{eq: lemma bound 2}, there exists a constant \(C>0\) such that
        \begin{equation*}
             \sup_{N\geq 1}\sup_{k\geq 1}\sum_{[k_1, \dots,k_N]_{k, 1}\setminus[k_1, \dots,k_N]_{k, 1}^3}{k\choose k_1,\dots,k_N}^{-1}\leq   \sup_{N\geq 1}\sup_{k\geq 1}\sum_{[k_1, \dots,k_N]_{k, 1}\setminus[k_1, \dots,k_N]_{k, 1}^3}\frac{(k-N+1)!}{k!} \leq C,
        \end{equation*}
 where in the final inequality we have used that the sum only has a non-zero number of summands when \(k\geq N\), and that the summand is \(o(k^{-1}N^{-2})\) as \(N\rightarrow \infty\) uniformly in \(k\geq N\). This and \eqref{eq: lemma bound 2} give \eqref{eq:step1}. Now, consider the partition of \([k_1,\dots,k_N]_{k}\) into the sets \([k_1,\dots,k_N]_{k,A},\) \(A\subseteq \{1,...,N\}\), where \([k_1,\dots,k_N]_{k,A}\) consists of elements that satisfy \(k_i=0 \) for \(i \in A\) and \(k_i> 0\) for \(i \notin A\). Using this partition and \eqref{eq:step1}, we conclude the proof with the following bound
 \begin{align*}
          \sup_{k\geq 2}\sum_{[k_1,\dots,k_N]_{k} }{k\choose k_1,\dots,k_N}^{-1}&= \sup_{k\geq 2}\sum_{A\subseteq \{1,...,N\}}\sum_{[k_1,\dots,k_N]_{k,A} }{k\choose k_1,\dots,k_N}^{-1}\\
          &=  \sum_{\alpha = 0}^{N}{N\choose\alpha}\sup_{k\geq 2}\sum_{[k_1, \dots, k_{N-\alpha}]_{k,1}} {k\choose k_1,\dots,k_{N-\alpha}}^{-1} \\
          & \le C\sum_{\alpha = 0}^{N}{N\choose\alpha}\leq 2^N,
        \end{align*} where in the second equality we use that the summand only depends on \(A\) through \(|A|\).
    \end{proof}
\section*{Acknowledgements}
EH and CBCD acknowledge the support of EPSRC grant MaThRad EP/W026899/2. 

\bibliographystyle{plain}
\bibliography{bibbp.bib}

\end{document}